\documentclass{amsart}
\usepackage{hyperref}
\hypersetup{pdftex, colorlinks, citecolor=magenta, linkcolor=blue}
\usepackage{amsthm}
\usepackage{enumitem}
\usepackage{amssymb}
\usepackage{bbm}
\usepackage[all, arc]{xy}
\usepackage[lite, numeric,sorted,traditional-quotes,compressed-cites]{amsrefs}
\newcommand{\id}[1]{{\text{id}}_{#1}}
\newcommand{\fraka}{{\mathfrak A}}
\newcommand{\Z}{{\mathbbm Z}}
\newcommand{\one}{{\mathbbm 1}}

\renewcommand{\id}[1]{{\rm id}_{#1}}

\newcommand{\trace}[1]{{\rm tr}(#1)}
\newcommand{\rend}[2]{{\rm End}_{#1}(#2)}

\newtheoremstyle{statement}% <name>
{13pt}% <Space above>
{13pt}% <Space below>
{\it}% <Body font>
{}% <Indent amount>
{\bf}% <Theorem head font>
{.$-$}% <Punctuation after theorem head>
{.5em}% <Space after theorem head>
{}% <Theorem head spec (can be left empty, meaning `normal')>

\theoremstyle{statement}

\newtheorem{theorem}{Theorem}[section]

\newtheorem{proposition}[theorem]{Proposition}

\newtheorem{cor}[theorem]{Corollary}
\newtheoremstyle{definition}% <name>
{13pt}% <Space above>
{13pt}% <Space below>
{}% <Body font>
{}% <Indent amount>
{\scshape}% <Theorem head font>
{.}% <Punctuation after theorem head>
{.5em}% <Space after theorem head>
{}% <Theorem head spec (can be left empty, meaning `normal')>

\theoremstyle{definition}

\makeatletter %this makes the section numbers appear in the left margin
\def\@seccntformat#1{\protect\makebox[0pt][r]{\@ifundefined{#1@cntformat}%
   {\csname the#1\endcsname.\quad}%
   {\csname #1@cntformat\endcsname}%
}}
\def\section@cntformat{\S\thesection.\ }
\def\subsection@cntformat{\S\thesubsection \ }

\makeatother
\begin{document}
\title{A remark on the additivity of traces in triangulated categories}

%    Remove any unused author tags.

%    author one information
\author{Shahram Biglari}
\address{Fakult\"at f\"ur Mathematik, Universit\"at Bielefeld, D-33615, Bielefeld, Germany}
\curraddr{} \email{biglari@mathematik.uni-bielefeld.de}
\thanks{}

%    author two information\author{}\address{}\curraddr{} \email{}



\date{2010}
%%%%%%%%%%%%%%%%%%%%%%%%%%%%%%%%%%%%%%%%%%%%%%%%%%%%%%%%%
%%%%%%%%%%%%%%%%%%%%%%%%%%%%%%%%%%%%%%%%%%%%%%%%%%%%%%%%%
\maketitle
\section{Introduction and statements}
In what follows a tensor category is understood to be an \emph{ACU} $\otimes$-category in the sense of Saavedra Rivano~\cite[Ch. I, 2.4.1]{Saavedra1972}. We denote the unit object by $\one$, the commutativity constraint by $\psi$, and the tensor structure by $\otimes$. There is also an associativity constraint that we omit and all these constraints are subject to natural compatibility conditions (\emph{loc. cit.} I, 2.4). Recall (Deligne~\cite[2.1.2]{Deligne1990}) that an object $X$ of a tensor category is said to be dualizable if there is an object $X^\vee$ and morphisms $\delta_X\colon \one\to X\otimes X^\vee$ and ${\rm ev}_X\colon X^\vee\otimes X\to \one$ such that the diagrams
\[
\xymatrix @!C=7pc{
{\mathbbm 1}\otimes X\ar[r]^-{\delta_X\otimes \id{X}}\ar[dr]_-{\psi_{{\mathbbm 1}, X}} & X\otimes X^{\vee}\otimes X \ar[d]^-{\id{X}\otimes {\rm ev}_X}& X^{\vee}\otimes X\otimes X^{\vee} \ar[d]_-{{\rm ev}_X\otimes\id{X^{\vee}}} & X^{\vee}\otimes{\mathbbm 1}\ar[l]_-{\id{X^{\vee}}\otimes \delta_{X}}\ar[dl]^-{{\psi_{X^{\vee},{\mathbbm 1}}}} &   \\ & X\otimes{\mathbbm 1} & {\mathbbm 1}\otimes X^{\vee}
}
\]
are commutative. For example for the tensor category of modules over a commutative ring, dualizability is (e.g. \emph{loc. cit.} 2.6) the same as being finitely generated and projective. With an appropriate interpretation, the morphism ${\rm ev}_X$ gives the trace. More concretely, let $X$ be a dualizable object and $f\colon X\to X$ an endomorphism. The trace of $f$, here denoted by $\trace{f;X}$, is defined to be the composite
\[
\one \xrightarrow{\delta_X} X\otimes X^\vee\xrightarrow{f\otimes \id{X^\vee}} X\otimes X^\vee\xrightarrow{\psi_{X, X^\vee}} X^\vee\otimes X\xrightarrow{{\rm ev}_X} \one.
\]
This is an element of $\rend{}{\one}$. The resulting map ${\rm tr}\colon \rend{}{X}\to \rend{}{\one}$ is linear. Moreover, when defined, the trace $\trace{f\otimes g;X\otimes Y}$ is the product of $\trace{f;X}$ and $\trace{g;Y}$. For the proofs of these and other properties see any of the references cited above.

We clarify some terminologies. A tensor category as above is (Mac Lane~\cite{Mac1963}) also called an (additive) symmetric monoidal category. A symmetric monoidal category in which each functor $Z\mapsto Z\otimes X$ has a right adjoint is ({Eilenberg-Kelly~\cite{MR0225841}) said to be closed. Recall the following result.

\begin{theorem}[May {\cite[0.1]{May2001}}]\label{thm:may} For any distinguished triangle $\Delta:X\to Z\to Y\to X[1]$ of dualizable objects in a closed symmetric monoidal category with a compatible triangulation we have
\[
\trace{\id{};Z}=\trace{\id{};X}+\trace{\id{};Y}.
\]
\end{theorem}
In what follows we let $D$ be a ${k}$-linear Karoubian (i.e. pseudo-abelian) rigid tensor triangulated category where $k=\bar{k}$ is an algebraically closed field of characteristic zero. Note that linearity means (\cite[Ch. I, 0.1.2]{Saavedra1972}) that $\rend{}{\one}$ is a $k$-algebra. Here the term \emph{rigid tensor triangulated} means a closed symmetric monoidal category with a compatible triangulation in the sense of~\cite{May2001} and in which every object is dualizable.  

An endomorphism $f=(f_X,f_Z,f_Y)$ of a distinguished triangle $\Delta$ in $D$ is a commutative diagram
\begin{equation}\label{def:end-triangle} 
    \begin{xy}*!C\xybox{% 
      \xymatrix{% 
	X\ar[r]\ar[d]^-{f_X} & Z\ar[r]\ar[d]^-{f_Z} & Y\ar[r]\ar[d]^-{f_Y} & X[1]\ar[d]^-{f_X[1]}\\
	X\ar[r] & Z\ar[r] & Y\ar[r] & X[1]
      }} 
    \end{xy}
\end{equation} 
with both rows being the given triangle $\Delta$. For example $\id{}=(\id{X},\id{Z},\id{Y})$ is an endomorphism of $\Delta$. The compositions of endomorphisms of triangles are defined in an obvious manner and is associative. We prove the following result.
\begin{proposition}\label{thm:add-torsion}
Let $f$ be an endomorphism of a distinguished triangle $X\to Z\to Y\to X[1]$ in $D$ with $f^n=\id{}$ for an integer $n> 0$. Then
\[
\trace{f_Z;Z}=\trace{f_X;X}+\trace{f_Y;Y}.
\]
\end{proposition}
\section{Proof}
Let $D$ and $k$ be as above. We prove a more general result than~\ref{thm:add-torsion}. Let $G$ be a group. A $G$-object in $D$ is a pair $(X, \rho)$ consisting of an object $X$ of $D$ and a $k$-algebra homomorphism $\rho:kG\to {\rm End}_\fraka(X)$ where $kG$ is the group algebra of $G$. We may denote $\rho(a)$ by $a_X$ or simply $a$. Let $Y$ be another $G$-object. An $G$-morphism or $G$-equivariant morphism from $X$ to $Y$ is a morphisms $f\colon X\to Y$ with $a_Yf=fa_X$ for all $a\in kG$. If $X$ is an $G$-object define the central function $$\chi_X\colon G\to {\rm End}_D({\mathbbm 1}),\quad g\mapsto \trace{g;X}.$$

We say that the distinguished triangle $\Delta$ is $G-$equivariant, if $X$, $Y$, and $Z$ are equipped with actions $\rho_X\colon G\to {\rm Aut}_D(Z)$ (similarly for $X$ and $Y$) and such that all morphisms (including the differential) are $G-$equivariant.

\begin{theorem}\label{thm:add-group}
If $G$ is torsion and $X\to Z\to Y\to X[1]$ is $G-$equivariant, then as functions $G\to {\rm End}_D({\mathbbm 1})$ we have
\[
\chi_Z=\chi_X+\chi_Y.
\]
\end{theorem}
\begin{proof}
We may assume that $G$ is finite. Let ${\rm Irr}kG$ be the set of isomorphism classes of irreducible $k-$representations of $G$. In $D$ we have a natural $G$-equivariant isomorphism
\begin{equation}\label{wr}
X\simeq \coprod_{V\in {\rm Irr}kG} V\otimes_k S_V(X)
\end{equation}
where $S_V(X)=\underline{{\rm Hom}}_{kG}(V, X)$ are certain objects and on which $G$ acts trivially. To see this, consider the contravariant functor $D\to (k-{\rm mod})$ given by
\[
{\rm Obj}(D)\ni Y\mapsto {\rm Hom}_{kG}\bigl(V, {\rm Hom}_{D}(Y, X)\bigr).
\]
This is representable. Indeed if in the above definition we replace $V$ by any finitely generated free $kG$-module $M$ and consider the corresponding functor, we see immediately that the functor is representable by an object $S_M(X)=$ a finite direct sum of $X$. The general case follows from this and the fact that $V$ is a finitely generated projective $kG$-module and hence the kernel (i.e. image) of a projector $\pi$ on a free $kG$-module $M$. Since $D$ is Karoubian, we can define $S_V(X)={\rm coker}(\pi^\ast)$ where $\pi^\ast\colon S_M(X)\to S_M(X)$ is induced by $\pi$. This is easily seen to represent $S_V(X)$. Once we have these objects, the decomposition of $X$ follows from the corresponding one for $kG$. It follows that the sequence \[S_V(X)\to S_V(Z)\to S_V(Y)\to S_V(X[1])\] being a direct summand of the original distinguished triangle is distinguished in $D$. Finally we note that by the above decomposition and $k$-linearity of trace we have
\begin{equation}\label{formula}
\trace{g, X}=\sum \chi_V(g)\trace{{\rm id};S_V(X)}
\end{equation}
where $\chi_V\colon G\to k$ is the usual character of $V$. Similarly for $Z$ and $Y$. The result follows from this and~\ref{thm:may}.
\end{proof}
\begin{proof}[{\sc Proof of~\ref{thm:add-torsion}}] Apply the result~\ref{thm:add-group} with $G=\Z/{n\Z}$ and the action $m\mapsto f_Z^{m}$ (resp. $m\mapsto f_X^{m}, m\mapsto f_Y^{m}$) on $Z$ (resp. $X, Y$).
\end{proof}
\section{Remark}
We conclude this short note by indicating a corollary of the proof of~\ref{thm:add-group}. We let $\fraka$ a Karoubian tensor category with $k\subseteq \rend{\fraka}{\one}$ where $k$ is an algebraic closure of ${\mathbbm Q}$. Define $\Z_\fraka$ to be the subring (=subgroup) of $\rend{\fraka}{\one}$ generated by all $\trace{\id{};X}$ with $X$ being dualizable in $\fraka$.
\begin{cor}
Let $f\colon X\to X$ be an endomorphism of a dualizable object in $\fraka$ with $f^n=\id{}$ for an integer $n> 0$. Then $\trace{f;X}\in \rend{\fraka}{\one}$ is integral over $\Z_\fraka$.
\end{cor}
\begin{proof}
Similar to the proof of~\ref{thm:add-torsion} consider $X$ with an action of $G=\Z/{n\Z}$. Note that in the category $\fraka$ the decomposition~\eqref{wr} and the formula~\eqref{formula} hold with exactly the same proof. Since the element $\chi_V(g)\in k$ is integral over $\Z$, the result follows from~\eqref{formula}.  
\end{proof}

\end{document}